\documentclass{amsart}
\usepackage{a4wide,amssymb}

\newcommand{\U}{\mathcal U}

\newcommand{\w}{\omega}

\newcommand{\dom}{\mathrm{dom}}

\newcommand{\supp}{\mathrm{supp}}
\newcommand{\IN}{\mathbb N}
\newcommand{\Tau}{\mathcal T}

\newtheorem{theorem}{Theorem}

\newtheorem{proposition}[theorem]{Proposition}

\newtheorem{example}{Example}

\newtheorem{claim}{Claim}

\title[{A metrizable Lawson semitopological semilattice with non-closed partial order}]{A metrizable Lawson semitopological semilattice\\ with non-closed partial order}
\author{Taras Banakh, Serhii Bardyla and Alex Ravsky}
\address{T.Banakh: Ivan Franko National University of Lviv (Ukraine) and Jan Kochanowski University in Kielce (Poland)}
\email{t.o.banakh@gmail.com}
\address{S.~Bardyla: Institute of Mathematics, Kurt G\"{o}del Research Center, Vienna (Austria)}
\email{sbardyla@yahoo.com}
\thanks{The second author was supported by the Austrian Science Fund FWF (Grant  I 3709-N35).}
\address{A.Ravsky: Department of Analysis, Geometry and Topology,
Pidstryhach Institute for Applied Problems of Mechanics and Mathematics
National Academy of Sciences of Ukraine}
\email{alexander.ravsky@uni-wuerzburg.de}

\begin{document}
\begin{abstract} We construct a metrizable Lawson semitopological semilattice $X$ whose partial order $\le_X=\{(x,y)\in X\times X:xy=x\}$ is not closed in $X\times X$.
This resolves a problem posed earlier by the authors.
\end{abstract}
\maketitle


In this paper we shall construct an example of a metrizable Lawson semitopological semilattice with non-closed partial order, thus answering a problem posed by the authors in \cite{BBR1}. 

A {\em semilattice} is a commutative semigroup $X$ whose any element $x\in X$ is an {\em idempotent} in the sense that $xx=x$. A typical example of a semilattice is any partially ordered set $X$ in which any finite non-empty set $F\subset X$ has the greatest lower bound $\inf(F)$. In this case the binary operation $X\times X\to X$, $(xy)\mapsto\inf\{x,y\}$, turns $X$ into a semilattice. 

Each semilattice $X$ carries a partial order $\le$ defined by $x\le y$ iff $xy=x$. For this partial order we have $xy=\inf\{x,y\}$.
A semilattice $X$ is called {\em $\uparrow$-finite} if for every element $x\in X$ its upper set ${\uparrow}x:=\{y\in X:x\le y\}$ is finite.

A ({\em semi\/}){\em topological semilattice} is a semilattice $X$ endowed with a topology such that the binary operation $X\times X\to X$, $xy\mapsto xy$, is (separately) continuous. A semitopological semilattice is {\em Lawson} if it has a base of the topology consisting of open subsemilattices.

The continuity of the semilattice operation in a Hausdorff topological semilattice implies the following well-known fact, see \cite[VI-1.14]{Bible}.

\begin{proposition} For any Hausdorff topological semilattice $X$ the partial order $$\le_X:=\{(x,y)\in X\times X:xy=y\}$$ is a closed subset of $X\times X$.
\end{proposition}

This proposition does not generalize to semitopological semilattices as shown by the following example constructed by the authors in \cite{BBR1}.

\begin{example} There exists a metrizable countable semitopological semilattice
 $X$ whose partial order is dense and non-closed in $X\times X$.
 \end{example}
 
At the end of the paper \cite{BBR1} the authors asked whether there exsists a Lawson Hausdorff semitopological semilattice $X$ with non-closed partial order. In this paper we shall give an affirmative answer to this question. Moreover, for any cardinal $\kappa$ we can construct a Lawson semitopological semilattice $X$ with non-closed partial order such that the topological space of $X$ is a $P_\kappa$-space.

A topological space $(X,\tau)$ is called a {\em $P_\kappa$-space} if for any family $\U\subset\tau$ of cardinality $|\U|\le\kappa$ the intersection $\bigcap\U$ belongs to the topology $\tau$. Each topological space is a $P_\kappa$-space for any finite cardinal $\kappa$. 

A topological space $X$ is called {\em zero-dimensional} if it has a base of the topology, consisting of clopen sets. A subset of a topological space is {\em clopen} if it is both closed and open. It is easy to see that every regular $P_\w$-space is zero-dimensional. The {\em weight} of a topological space $(X,\tau)$ is the smallest cardinality of a base of the topology $\tau$.
 
The following example (answering Problem 1 in \cite{BBR1}) is the main result of this paper.
 
\begin{example}\label{e:main} For any infinite cardinal $\lambda$ there exists a Lawson semitopological semilattice $X$ having the following properties:
\begin{enumerate}
\item the partial order $\{(x,y)\in X\times X:xy=x\}$ of $X$ is not closed in $X\times X$;
\item the semilattice $X$ is $\uparrow$-finite;
\item the cardinality and the weight of the space $X$ both are equal to $\lambda$;
\item $X$ is a Hausdorff zero-dimensional space;
\item $X$ is a $P_\kappa$-space for any cardinal $\kappa<\mathrm{cf}(\lambda)$;
\item if $\lambda=\w$, then the countable space $X$ is  metrizable.
\end{enumerate}
\end{example} 
 
\begin{proof} We identify the cardinal $\lambda$ with the set $[0,\lambda)$ of ordinals, smaller than $\lambda$.  For two ordinals $\alpha<\beta$ in $\lambda$ by $[\alpha,\beta)$ we denote the oerder-interval consisting of ordinals $\gamma$ such that $\alpha\le\gamma<\beta$. 

Consider the semilattice $\{0,1,2\}$, endowed with the operation of minimum. In its power $\{0,1,2\}^\lambda$ consider the subsemilattice $X$ consisting of the functions $x:\lambda\to\{0,1,2\}$ having finite support $\supp(x):=x^{-1}(\{0,1\})$. 

For a function $x\in X$ let $\|x\|$ denote the smallest ordinal $\alpha\in\lambda$ such that $\supp(x)\cap[\alpha,\lambda)=\emptyset$. If $\supp(x)\ne \emptyset$, then $\|x\|=\max(\supp(x))+1$ and hence the ordinal  $\|x\|-1$ is well-defined and equals $\max(\supp(x))$.

The following property of the semilattice $X$ easily follows from the definition of $X$.

\begin{claim}\label{cl1} The semilattice $X$ has cardinality $|X|=\lambda$ and is $\uparrow$-finite. More precisely, for every $x\in X$ its upper set ${\uparrow}x$ has cardinality $|{\uparrow}x|\le 3^{|\supp(x)|}$.
\end{claim}

Now we shall define a topology $\tau$ on $X$ turning the semilattice $X$ into a Lawson semitopological semilatice with non-closed partial order.

Write the semilattice $X$ as the union $X_0\cup X_1\cup X_2$ of the subsemilattices $$X_i:=\{x\in X:x(0)=i\}\mbox{ \  for \ }i\in\{0,1,2\}.$$

For every function $x\in X$ and ordinal $\alpha\in[\|x\|,\lambda)$ we shall define a subsemilattice $V_\alpha(x)$ of $X$ as follows.

If $x\in X_2$, then we put $V_\alpha(x):=\{x\}$.

If $x\in X_1$, then we put
$$
\begin{aligned}
V_\alpha(x):=&\big\{y\in X_1:y{\restriction}[1,\alpha)=x{\restriction}[1,\alpha)\mbox{ \ and \ }y([\alpha,\lambda))\subset\{0,2\}\big\}\cup\\
&\big\{y\in X_2:y{\restriction}[1,\alpha)=x{\restriction}[1,\alpha),\;\;y([\alpha,\lambda))\subset \{0,2\}\mbox{ \ and \ }\|y\|>\alpha\big\}.
\end{aligned}
$$

If $x\in X_0$, then we put
$$
\begin{aligned}
V_\alpha(x):=&\big\{y\in X_0:y{\restriction}[1,\alpha)=x{\restriction}[1,\alpha)\big\}\cup\\
&\big\{y\in X_2:y{\restriction}[1,\alpha)=x{\restriction}[1,\alpha),\;\|y\|>\alpha\mbox{  and }y(\|y\|-1))=1\big\}.
\end{aligned}
$$ It is easy to see that $x\in V_\beta(x)\subset V_{\alpha}(x)$ for any $x\in X$ and ordinals $\alpha,\beta\in\lambda$ with $\|x\|\le\alpha\le\beta$. 

Endow $X$ with the topology $\tau$ consisting of the sets $U\subset X$ such that  for any $x\in U$ there exists an ordinal $\alpha\in[\|x\|,\lambda)$ such that $V_\alpha(x)\subset U$.

\begin{claim}\label{cl2} For every $x\in X$ and $\alpha\in[\|x\|,\lambda)$, the set $V_\alpha(x)$ is open in the topological space $(X,\tau)$.
\end{claim}

\begin{proof} Given any element $y\in V_\alpha(x)$, we should find an ordinal $\beta\in[\|y\|,\lambda)$ such that $V_\beta(y)\subset V_\alpha(x)$. Choose any ordinal $\beta\in[\alpha,\lambda)$ such that $\beta\ge\|y\|$. If $y\in X_2$, then $V_\beta(y)=\{y\}\subset V_\alpha(x)$ and we are done. So, we assume that $y\notin X_2$. In this case $x\notin X_2$.

If $x\in X_1$, then $y\in V_\alpha(x)\setminus X_2\subset X_1$ and hence $y{\restriction}[1,\alpha)=x{\restriction}[1,\alpha)$ and $y([\alpha,\lambda))\subset\{0,2\}$. To prove that $V_\beta(y)\subset V_\alpha(x)$, take any $z\in V_\beta(y)$. The definition of $V_\beta(y)$ for $y\in X_1$ ensures that $z\in X_1\cup X_2$. If $z\in X_1$, then $z{\restriction}[1,\beta)=y{\restriction}[1,\beta)$ and $z([\beta,\lambda))\subset\{0,2\}$. Taking into account that $\alpha\le\beta$, we conclude that $z{\restriction}[1,\alpha)=y{\restriction}[1,\alpha)=x{\restriction}[1,\alpha)$ and $$z([\alpha,\lambda))=z([\alpha,\beta))\cup z([\beta,\lambda))=y([\alpha,\beta))\cup z([\beta,\lambda)\subset y([\alpha,\lambda))\cup z([\beta,\lambda))\subset\{0,2\},$$witnessing that $z\in V_\alpha(x)\cap X_1$.

If $z\in X_2$, then the definition of $V_\beta(y)$ ensures that $z{\restriction}[1,\beta)=y{\restriction}[1,\beta)$, $z([\beta,\lambda))\subset \{0,2\}$ and $\|z\|>\beta$. Then $z{\restriction}[1,\alpha)=y{\restriction}[1,\alpha)=x{\restriction}[1,\alpha)$, $z([\alpha,\lambda))=z([\alpha,\beta))\cup z([\beta,\lambda))=y([\alpha,\beta))\cup z([\beta,\lambda))\subset \{0,2\}$ and $\|z\|>\beta\ge\alpha$, witnessing that  $z\in V_\alpha(x)\cap X_2$.
\smallskip

Next, assume that $x\in X_0$. In this case $y\in V_\alpha(x)\setminus X_2\subset X_0$ and $y{\restriction}[1,\alpha)=x{\restriction}[1,\alpha)$. To prove that $V_\beta(y)\subset V_\alpha(x)$, take any $z\in V_\beta(y)$. The definition of $V_\beta(y)$ ensures that $z\in X_0\cup X_2$. If $z\in X_0$, then $z{\restriction}[1,\alpha)=y{\restriction}[1,\alpha)=x{\restriction}[1,\alpha)$ and $z\in V_\alpha(x)$ by the definition of $V_\alpha(x)$. If $z\in X_2$, then $z{\restriction}[1,\beta)=y{\restriction}[1,\beta)$, $\|z\|>\beta$ and $z(\|z\|-1)=1$. Then $z{\restriction}[1,\alpha)=y{\restriction}[1,\alpha)=x{\restriction}[1,\alpha)$, $\|z\|>\beta\ge\alpha$ and $z(\|z\|-1)=1$, witnessing that $z\in V_\alpha(x)$.
\end{proof}

Claim~\ref{cl2} implies

\begin{claim}\label{cl3} The family $\mathcal B:=\{V_\alpha(x):x\in X,\;\alpha\in[\|x\|,\lambda)\}$ is a base of the topology $\tau$.
\end{claim}

\begin{claim}\label{cl4} The weight $w(X,\tau)$ of the topological space $(X,\tau)$  equals $\lambda$.
\end{claim}

\begin{proof} Claim~\ref{cl3} implies that $w(X,\tau)\le|\mathcal B|\le\lambda$. Taking into account that $X_2$ is a discrete subspace of $(X,\tau)$, we conclude that $w(X,\tau)\ge|X_2|=\lambda$.
\end{proof}

The following claim implies that the topological space $(X,\tau)$ is zero-dimensional.

\begin{claim}\label{cl5} For every $x\in X$ and a non-zero ordinal $\alpha\in[\|x\|,\lambda)$ the set $V_\alpha(x)$ is closed in the topological space $(X,\tau)$.
\end{claim}

\begin{proof} Given any $y\in X\setminus V_\alpha(x)$, we should find an ordinal $\beta\in[\|y\|,\lambda)$ such that $V_\beta(y)\cap V_\alpha(x)=\emptyset$. We claim that the ordinal $\beta=\max\{\alpha,\|y\|\}$ has the desired property.

Six cases are possible.
\smallskip

1) If $y\in X_2$, then $V_\alpha(x)\cap V_\beta(y)=V_\alpha(x)\cap\{y\}=\emptyset$ and we are done. 
\smallskip

2) If $y\in X_0\cup X_1$ and $x\in X_2$, then $V_\alpha(x)\cap V_\beta(y)=\{x\}\cap V_\beta(y)=\emptyset$ as $\|x\|\le\alpha\le \beta$.
\smallskip

3) $x,y\in X_0$. In this case $y\in X_0\setminus V_\alpha(x)$ implies $y{\restriction}[1,\alpha)\ne x{\restriction}[1,\alpha)$ and hence $V_\alpha(x)\cap V_\beta(y)\subset V_\alpha(x)\cap V_\alpha(y)=\emptyset$.
\smallskip

4) $x,y\in X_1$. In this case $y\in X_1\setminus V_\alpha(x)$ implies $y{\restriction}[1,\alpha)\ne x{\restriction}[1,\alpha)$ or $y([\alpha,\lambda))\not\subset\{0,2\}$. If $y{\restriction}[1,\alpha)\ne x{\restriction}[1,\alpha)$, then $V_\alpha(x)\cap V_\beta(y)\subset V_\alpha(x)\cap V_\alpha(y)=\emptyset$. If $y{\restriction}[1,\alpha)=x{\restriction}[1,\alpha)$, then $y([\alpha,\lambda))\not\subset\{0,2\}$ and the choice of the ordinal $\beta\ge\|y\|$ guarantees that $y([\alpha,\beta))\not\subset\{0,2\}$. Then for any $z\in V_{\beta}(y)$ we get $z([\alpha,\beta))=y([\alpha,\beta))\not\subset\{0,2\}$ and hence $z\notin V_\alpha(x)$, which implies $V_\alpha(x)\cap V_\beta(y)=\emptyset$.
\smallskip

5) $x\in X_0$ and $y\in X_1$. In this case for any $z\in V_\alpha(x)\cap V_\beta(y)$, we have $z\in X_2$. The inclusion $z\in X_2\cap V_\beta(y)$ implies $\|z\|>\beta$ and $z([\beta,\lambda))\subset\{0,2\}$. On the other hand, $z\in X_2\cap V_\alpha(x)$ implies $1=z(\|z\|-1)\in z([\beta,\lambda))\subset\{0,2\}$, which is a desired contradiction implying that $V_\alpha(x)\cap V_\beta(y)=\emptyset$.
\vskip3pt

6) $x\in X_1$ and $y\in X_0$.  In this case for any $z\in V_\alpha(x)\cap V_\beta(y)$, we have $z\in X_2$. The inclusion $z\in X_2\cap V_\alpha(x)$ implies $z([\alpha,\lambda))\subset \{0,2\}$ and the inclusion $z\in X_2\cap V_\beta(y)$ implies $\|z\|>\beta$ and $1=z(\|z\|-1)\in z([\beta,\lambda))\subset z([\alpha,\lambda))\subset\{0,2\}$, which is a desired contradiction implying that $V_\alpha(x)\cap V_\beta(y)=\emptyset$.
\end{proof}

\begin{claim} The topological space $(X,\tau)$ is Hausdorff.
\end{claim}

\begin{proof} Given two distinct points $x,y\in X$, put $\alpha:=\max\{\|x\|,\|y\|\}$ and consider four possible cases.
\smallskip

1) If $x\in X_2$, then by Claims~\ref{cl2}, \ref{cl5}, $O_x:=\{x\}$ and $O_y:=X\setminus\{x\}$ are disjoint clopen neighborhoods of the points $x,y$, respectively.
\smallskip

2) If $y\in X_2$, then $O_x:=X\setminus \{y\}$ and $O_y:=\{y\}$ are  disjoint clopen neighborhoods of the points $x,y$, respectively.
\smallskip

3) If $x,y\in X_0$ or $x,y\in X_1$, then $x\ne y$ impies $x{\restriction}[1,\alpha)\ne y{\restriction}[1,\alpha)$. Consequently, $V_\alpha(x)$ and $V_\alpha(y)$ are disjoint clopen neighborhoods of the points $x,y$, respectively.
\smallskip

4) If the doubleton $\{x,y\}$ intersects both sets $X_0$ and $X_1$, then $O_x:=V_\alpha(x)$ and $O_y:=X\setminus V_\alpha(x)$ are disjoint clopen neighborhoods of the points $x,y$, respectively. 
\end{proof}

\begin{claim} For any cardinal $\kappa<\mathrm{cf}(\lambda)$, the space $(X,\tau)$ is a $P_\kappa$-space.
\end{claim}

\begin{proof} Given any family $\U\subset\tau$ of cardinality $|\U|\le\kappa<\mathrm{cf}(\lambda)$, we should prove that the intersection $\bigcap\U$ belongs to the topology $\tau$. Fix any point $x\in\bigcap\U$. By the definition of the topology $\tau$, for every set $U\in\U\subset\tau$ there exists an ordinal $\alpha_U\in[\|x\|,\lambda)$ such that $V_{\alpha_U}(x)\subset U$. Since $|\U|\le\kappa<\mathrm{cf}(\lambda)$, the ordinal $\alpha=\sup\{\alpha_U:U\in\U\}$ is strictly smaller than $\lambda$. Since $V_\alpha(x)\subset\bigcap\U$, the set $\bigcap\U$ belongs to the topology $\tau$ by the definition of $\tau$.
\end{proof}

\begin{claim} If $\lambda=\w$, then the countable space $(X,\tau)$ is metrizable.
\end{claim}

\begin{proof} Being Hausdorff and zero-dimensional, the space $(X,\tau)$ is regular. If $\lambda=\w$, then by Claim~\ref{cl4}, the space $(X,\tau)$ is second-countable. By the Urysohn Metrization Theorem 4.2.9 \cite{En}, the space $(X,\tau)$ is metrizable.
\end{proof}

\begin{claim} $(X,\tau)$ is a Lawson semitopological semilattice.
\end{claim}

\begin{proof} Given any element $a\in X$, we first prove the continuity of the shift $s_a:X\to X$, $s_a:x\mapsto ax$, at any point $x\in X$. Given any neighborhood $O_{ax}\in\tau$ of $ax$, we need to find a neighborhood $O_x\in\tau$ of $x$ such that $aO_x\subset O_{ax}$. Using Claim~\ref{cl2}, find an ordinal $\alpha\in\lambda$ such that $\alpha\ge\|ax\|=\max\{\|a\|,\|x\|\}$ and $V_{\alpha}(ax)\subset O_{ax}$.
It remains to prove that $aV_\alpha(x)\subset V_{\alpha}(ax)\subset O_{ax}$.
\smallskip

Four cases are possible.
\smallskip

1) $x\in X_2$. In this case $aV_\alpha(x)=a\{x\}=\{ax\}\subset V_{\alpha}(ax)$.
\smallskip

2) $a\in X_0$ or $x\in X_0$. In this case $ax\in X_0$ and for any $z\in V_\alpha(x)$ we have $z{\restriction}[1,\alpha)=x{\restriction}[1,\alpha)$, which implies $az{\restriction}[1,\alpha)=ax{\restriction}[1,\alpha)$ and finally $az\in V_\alpha(ax)\cap X_0$. Therefore, $aV_\beta(x)\subset V_\alpha(ax)$.
\smallskip

3) $a\in X_1$ and $x\in X_1$. In this case $ax\in X_1$ and for any $z\in V_\alpha(x)$ we have $z{\restriction}[1,\alpha)=x{\restriction}[1,\alpha)$ and $z([\alpha,\lambda))\subset\{0,2\}$. Taking into account that $\|a\|\le \alpha$, we conclude that $az{\restriction}[1,\alpha)=ax{\restriction}[1,\alpha)$ and $(az)([\alpha,\lambda))=z([\alpha,\lambda))\subset\{0,2\}$, which implies $az\in V_\alpha(ax)\cap X_1$ and $aV_\alpha(x)\subset V_\alpha(ax)\subset O_{ax}$.
\vskip3pt

4) $a\in X_2$ and $x\in X_1$. In this case $ax\in X_1$ and $V_\alpha(x)\subset X_1\cup X_2$. Observe that for any $z\in V_\alpha(x)$ we have $z{\restriction}[1,\alpha)=x{\restriction}[1,\alpha)$ and $z([\alpha,\lambda))\subset\{0,2\}$. Taking into account that $\|a\|\le \alpha$, we conclude that $az{\restriction}[1,\alpha)=ax{\restriction}[1,\alpha)$ and $(az)([\alpha,\lambda))=z([\alpha,\lambda))\subset\{0,2\}$. If $z\in X_1$, then $az\in X_1$ and $az\in V_\alpha(ax)\cap X_1$.
If $z\in X_2$, then $z\in V_\alpha(x)$ implies $\|z\|>\alpha$ and then $\|az\|\ge \|z\|>\alpha$ and $az\in V_\alpha(az)\cap X_2$. In both cases we get $az\in V_\alpha(ax)$ and hence $aV_\alpha(x)\subset V_\alpha(ax)\subset O_{ax}$.
\smallskip

Therefore, $(X,\tau)$ is semitopological semilattice. Since the base $\mathcal B=\{V_\alpha(x):x\in X,\;\alpha\in[\|x\|,\lambda)\}$ of the topology $\tau$ consists of open subsemilattices of $X$, the semitopological semilattice $(X,\tau)$ is Lawson.
\end{proof}

\begin{claim}\label{cl10} The partial order $\le_X:=\{(x,y)\in X\times X:xy=x\}$ of $X$ is not closed in $X\times X$.
\end{claim} 

\begin{proof} For every $\alpha\in\w$ consider the elements $\mathbf 0_\alpha:\lambda\to\{0,2\}$ and $\mathbf 1_\alpha:\lambda\to\{1,2\}$ of $X$, uniquely determined by the conditions $\mathbf 0_\alpha^{-1}(0)=\{\alpha\}=\mathbf 1^{-1}_\alpha(1)$. It is clear that $\mathbf 0_\alpha\le \mathbf 1_\alpha$ and hence $(\mathbf 1_0,\mathbf 0_0)\notin\le_X$. On the other hand, for any ordinal $\alpha\in[1,\lambda)$ we have 
$$\{(\mathbf 0_\beta,\mathbf 1_\beta):\beta\in[\alpha,\lambda)\}\subset \big(V_\alpha(\mathbf 1_0)\times V_\alpha(\mathbf 0_0)\big)\cap \big(\le_X\big),$$which means that the pair $(\mathbf 1_0,\mathbf 0_0)\notin\le_X$ belongs to the closure of $\le_X$. Consequently, the partial order $\le_X$ of $X$ is not closed in $X\times X$.
\end{proof}
Claims~\ref{cl1}--\ref{cl10} imply that $(X,\tau)$ is a Lawson semitopological semilattice possessing the properties (1)--(6) of Example~\ref{e:main}.
\end{proof}
 

\end{document}